\documentclass[11pt]{article}

\usepackage[T1]{fontenc}
\usepackage[utf8]{inputenc}
\usepackage{lmodern}
\usepackage[a4paper,margin=31mm]{geometry}
\usepackage{amsmath,amssymb,amsthm,mathtools,mathrsfs}
\usepackage{booktabs,array}
\usepackage{enumitem}
\usepackage{graphicx}
\usepackage{flafter}
\usepackage{xcolor}
\usepackage{microtype}
\usepackage{hyperref}

\hypersetup{
  colorlinks=true,
  linkcolor=blue!55!black,
  citecolor=blue!55!black,
  urlcolor=blue!55!black,
  pdftitle={A Dyadic Frequency Law for a Perturbed Hofstadter Q-Recursion},
  pdfauthor={Marco Mantovanelli},
  pdfsubject={Value frequencies and dyadic structure in a nested recurrence},
  pdfkeywords={nested recurrence, Hofstadter recursion, value frequency,
    dyadic self-similarity, binary words, plane trees, Euler product}
}

\newtheorem{theorem}{Theorem}[section]
\newtheorem{proposition}[theorem]{Proposition}
\newtheorem{lemma}[theorem]{Lemma}
\newtheorem{corollary}[theorem]{Corollary}
\theoremstyle{definition}

\newtheorem{example}[theorem]{Example}
\theoremstyle{remark}
\newtheorem{remark}[theorem]{Remark}

\newcommand{\Bcal}{\mathcal B}
\newcommand{\Tcal}{\mathcal T}
\newcommand{\Acore}{\mathsf A}
\newcommand{\Ccore}{\mathsf C}
\newcommand{\Cword}{\mathcal C}
\newcommand{\Deg}{\operatorname{Deg}}
\newcommand{\Sh}{\operatorname{Sh}}
\newcommand{\ordtwo}{\nu_2}
\newcommand{\rev}{\operatorname{rev}}
\newcommand{\Interleave}{\operatorname{Interleave}}
\newcommand{\ms}[1]{\{\!\{#1\}\!\}}
\newcommand{\mplus}{\mathbin{\uplus}}

\title{A Dyadic Frequency Law for a Perturbed Hofstadter $Q$-Recursion}
\author{Marco Mantovanelli\\[.35em]
\small\href{mailto:Marco@Mantovanelli.de}{\texttt{Marco@Mantovanelli.de}}\\[-.1em]
\small\href{https://orcid.org/0009-0002-0631-293X}
{ORCID: 0009-0002-0631-293X}}
\date{}

\begin{document}
\maketitle

\begin{abstract}
We study the perturbed Hofstadter $Q$-recursion
\[
 Q(1)=Q(2)=1,
 \qquad
 Q(n)=Q\bigl(n-Q(n-1)\bigr)+Q\bigl(n-Q(n-2)\bigr)+(-1)^n.
\]
Cloitre proved that the recursion is globally well-defined and encoded its
odd- and even-indexed subsequences by exact binary arches and canonical plane
forests.  Since every value is odd, let
\[
 F(s)=\#\{n\geq 1:Q(n)=2s-1\}.
\]
We prove, for every \(k\geq0\), the multiset identity
\[
 \ms{F(s):2^k\leq s<2^{k+1}}
 =
 \ms{3+\ordtwo(j):1\leq j\leq2^k}.
\]
The proof expresses frequencies as sums of plateau lengths. Folding the
associated gaps identifies the excess frequencies \(F(s)-3\) with the
outdegrees of a canonical plane tree. For \(k\geq2\) the two highest
frequencies, \(k+3\) and \(k+2\), occur at consecutive levels beginning
at \(\lceil 2^{k+2}/3\rceil\). As a corollary, the block sum with marks
\(z^{F(s)-3}\) and weights \(2^{-ku}\), multiplied by the Dirichlet eta
function, has the Euler product of \(\zeta(u)\) with its factor at \(2\)
replaced by \((1-z2^{-u})^{-1}\).
Finally, the dyadic defect
\(Q(2n)-2Q(n)\) attains order \(n/\sqrt{\log n}\) in absolute value
along a subsequence.
\end{abstract}



\medskip
\noindent\textbf{Version note.}
Version 4 locates the two highest frequencies in each block with at least
two levels, derives the marked Euler product, shortens the proof of the
central-pair lemma, and proves the sharp order of the dyadic defect along
a subsequence. Version 3
introduced the folded-degree proof of the frequency law and removed two
incorrect ancillary claims from version 2.

\section{Introduction}

In nested recurrences, previously computed values are reused to determine the arguments of subsequent recursive calls.  This makes even global existence delicate.  The classical
Hofstadter \(Q\)-recursion is the most familiar example
\cite{hofstadter1979}; related meta-Fibonacci recurrences often become
tractable only after a combinatorial encoding by trees or frequency words
\cite{conolly1989,tanny1992,balamohan2008,jackson2006,deugau2006,
ruskeydeugau2009,fox2022}.

The recurrence studied here was discovered by the author during a
computational investigation of perturbations of Hofstadter's
recursions.  It was first circulated in the first version of this
manuscript and is recorded as sequence A394051 in the OEIS
\cite{oeisA394051}. It is defined by
\begin{equation}
 Q(1)=Q(2)=1,
 \qquad
 Q(n)=Q\bigl(n-Q(n-1)\bigr)+Q\bigl(n-Q(n-2)\bigr)+(-1)^n
 \quad(n\geq3).
 \label{eq:recurrence}
\end{equation}
Cloitre proved global well-definedness by a direct combinatorial
construction \cite{cloitre2026mantovanelli}. An independent proof uses a
computer-assisted induction on return words \cite{mantovanelli2026finite}.
Cloitre's odd--even decomposition produces
two slow sequences with increments in \(\{0,1\}\), alternating positive and
negative binary arches, exact-fit interleavings, Dyck words, and canonical
plane forests.  Among other results, that work proves
\[
 Q(n)=\frac n2+O\!\left(\frac{n}{\sqrt{\log n}}\right)
\]
and proves that the error is not
\(o\bigl(n/\sqrt{\log n}\bigr)\).

Cloitre's work determines the global orbit and its binary-forest structure
but does not determine the value-frequency distribution.  The present paper
extracts a new local-time invariant from that structure.  Its main new
ingredient is the exact folding of reflected gap degrees across a dyadic cut,
including the central-pair boundary analysis.  Every term of \(Q\) is odd,
so we define the frequency of the
value \(2s-1\) by
\[
 F(s)=\#\{n\geq1:Q(n)=2s-1\}.
\]
Our main result is the following.

\begin{theorem}[Dyadic frequency law]
\label{thm:intro-frequency}
For every integer \(k\geq0\),
\begin{equation}
 \ms{F(s):2^k\leq s<2^{k+1}}
 =
 \ms{3+\ordtwo(j):1\leq j\leq2^k}.
 \label{eq:main-law-intro}
\end{equation}
Here \(\ordtwo(j)\) is the exponent of \(2\) in \(j\), and both sides of
\eqref{eq:main-law-intro} are multisets.
\end{theorem}

Thus \(3+h\) occurs \(2^{k-h-1}\) times for \(0\leq h<k\), while
\(k+3\) occurs once.  The order is not specified.  For example,
\begin{equation}
 (F(4),F(5),F(6),F(7))=(3,3,5,4),
 \label{eq:not-pointwise}
\end{equation}
whereas \((3+\ordtwo(j))_{j=1}^4=(3,4,3,5)\). The two highest entries do
have prescribed positions. For every \(k\geq2\), the maximum occurs at
\(\lceil 2^{k+2}/3\rceil\) and the second highest frequency occurs at
the next level (Corollary~\ref{cor:block-peaks}).

One consequence of the multiset law uses a common weight \(2^{-ku}\) for
all levels in block \(k\). Marking the excess at level \(m\) by
\(z^{F(m)-3}\) gives the identity of Corollary~\ref{cor:marked-euler},
\[
 \eta(u)\sum_{k\geq0}2^{-ku}
   \sum_{m=2^k}^{2^{k+1}-1}z^{F(m)-3}
 =\frac{1-2^{-u}}{1-z2^{-u}}\,\zeta(u)
 \qquad\bigl(\operatorname{Re}u>1,\ |z|<2^{\operatorname{Re}u}\bigr).
\]
Here \(\eta(u)=(1-2^{1-u})\zeta(u)\). The product depends only on the
block multisets. The ordinary marked Dirichlet series
\(\sum_{m\ge1}z^{F(m)-3}m^{-u}\) retains the individual level weights.
For example, its coefficient at \(m=2\) is \(1\) since \(F(2)=3\).
The corresponding coefficient of the product is \(z\).

The proof of Theorem~\ref{thm:intro-frequency} proceeds in four steps.
\begin{enumerate}[label=(\roman*),leftmargin=2.2em,beginpenalty=10000]
\item express \(F(s)\) as the sum of two plateau lengths;
\item encode every half-arch by a balanced binary core and fold its zero-gap
      degrees under reflection;
\item prove that each reflected degree pair has exactly one positive member,
      with ordered central pair \((0,1)\) or \((1,0)\);
\item identify the folded degrees with the children of one canonical tree.
\end{enumerate}
The central-pair assertion is the boundary ingredient that makes the dyadic
cut exact.  We prove it explicitly in Lemma~\ref{lem:paired-degrees}.

The logical dependence of the proof and its consequences is
\begingroup\small
\[
 \begin{gathered}
 \text{Cloitre: binary arches and canonical forests}\\[-.5em]
 \Downarrow\\[-.5em]
 \text{plateau identity and ordered gap formulas}\\[-.5em]
 \Downarrow\\[-.5em]
 \text{reflection pairs and the ordered central pair}\\[-.5em]
 \Downarrow\\[-.5em]
 \text{folded block}=\Deg(\Tcal_k)\\[-.5em]
 \Downarrow\\[-.5em]
 \text{dyadic frequency law}\\[.15em]
 \text{Consequences only: frequency law}\Longrightarrow
 \text{block mass}\Longrightarrow\text{dyadic hitting time}.
 \end{gathered}
\]
\endgroup
In particular, neither block mass nor the hitting-time formula is used in the
frequency proof.

Figure~\ref{fig:first-values} shows the orbit on an initial range.
Section~\ref{sec:clock} uses adjacent values of the dyadic defect to
transfer the known arch amplitude to a lower bound for that defect.
Section~\ref{sec:seeds} treats alternative initial triples.

\begin{figure}[t]
  \centering
  \includegraphics[width=.88\textwidth]{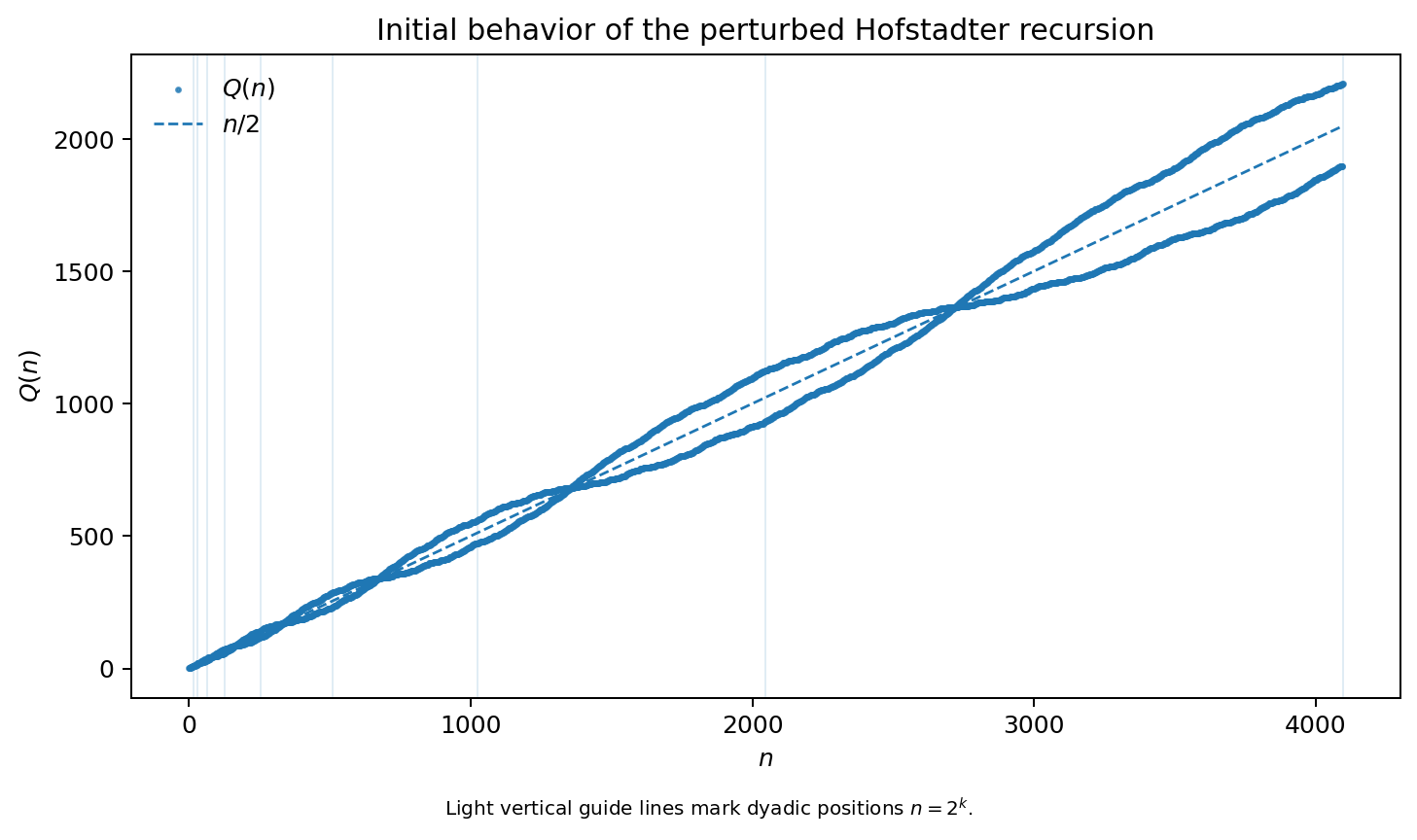}
  \caption{The canonical orbit on an initial range.  The proof below uses the
  exact binary arch construction, not visual regularity in this plot.}
  \label{fig:first-values}
\end{figure}

\section{Basic properties and dyadic geometry}
\label{sec:basic-properties}

\subsection{First values, parity, and the two slow branches}

The first terms are
\begin{equation}
\begin{aligned}
 &1,1,1,3,3,3,5,5,5,7,5,9,7,9,7,11,9,11,11,\\
 &11,13,11,15,13,15,13,17,13,19,15,19,17,\ldots .
\end{aligned}
\label{eq:first-terms}
\end{equation}
In particular, the orbit itself is neither monotone nor slowly increasing.
The slow objects are its normalized odd- and even-indexed branches:
\begin{equation}
 U(m)=\frac{Q(2m-1)+1}{2},
 \qquad
 V(m)=\frac{Q(2m)+1}{2}.
 \label{eq:UV}
\end{equation}

\begin{proposition}[Fundamental properties]
\label{prop:fundamental-properties}
The recurrence \eqref{eq:recurrence} is globally well-defined, with legal
earlier recursive arguments and
\[
 1\leq Q(n)\leq n \qquad(n\geq1).
\]
Every value \(Q(n)\) is odd.  The sequences \(U,V\) are unbounded and
nondecreasing, and
\begin{equation}
 \Delta U(m),\Delta V(m)\in\{0,1\}.
 \label{eq:basic-binary-increments}
\end{equation}
Consequently every positive odd integer occurs in \(Q\), each such value has
finite frequency, and
\[
 \lim_{n\to\infty}\frac{Q(n)}n=\frac12.
\]
\end{proposition}

\begin{proof}
Global construction, legality, and the stated bounds are contained in
Cloitre's Theorem~8.4 \cite{cloitre2026mantovanelli}. For the independent
return-word proof of global existence, see \cite{mantovanelli2026finite}. Once global
definition is known, oddness is immediate by induction: the two recursive
values are odd, their sum is even, and adding \((-1)^n\) gives an odd integer.
Cloitre's Lemma~7.1, together with the arch endpoint identities in his
equation~(2), proves the assertions about \(U,V\).  Since each branch starts
at one, is unbounded, and advances only by zero or one, it assumes every
positive integer and eventually leaves each fixed level.  This proves the
assertions about the values and their frequencies.  The limit follows from
Cloitre's Theorem~2.1(i), or from the stronger estimate in
Theorem~2.1(iii),
\[
 Q(n)=\frac n2+O\!\left(\frac{n}{\sqrt{\log n}}\right).
\]
\end{proof}

The raw first differences nevertheless have a simple exact interpretation.
Put
\[
 \Delta_Q(n)=Q(n+1)-Q(n),
 \qquad
 \delta(m)=V(m)-U(m),
 \qquad
 \varepsilon_m=U(m+1)-U(m).
\]
Then
\begin{equation}
 \Delta_Q(2m-1)=2\delta(m),
 \qquad
 \Delta_Q(2m)=2\bigl(\varepsilon_m-\delta(m)\bigr).
 \label{eq:first-difference-branches}
\end{equation}
Thus the two visible branches in Figure~\ref{fig:first-differences} are the
same alternating arch walk, once directly and once reflected with a one-bit
correction.  This is not a bounded-increment property of \(Q\): already
\(\Delta_Q(11)=Q(12)-Q(11)=4\).  In fact, if \(\mathcal H_r\) is the maximum height
of \(\delta\) on the \(r\)-th positive arch, Cloitre's exact amplitude law is
\begin{equation}
 \mathcal H_r=2+\sum_{j=1}^{r}\binom{2j+1}{j}
 \sim \frac{8}{3\sqrt{\pi}}\frac{4^r}{\sqrt r},
 \label{eq:arch-amplitude-overview}
\end{equation}
so \(\Delta_Q\) is unbounded \cite[Theorem~2.1(ii)]{cloitre2026mantovanelli}.

\begin{figure}[t]
  \centering
  \includegraphics[width=.94\textwidth]
    {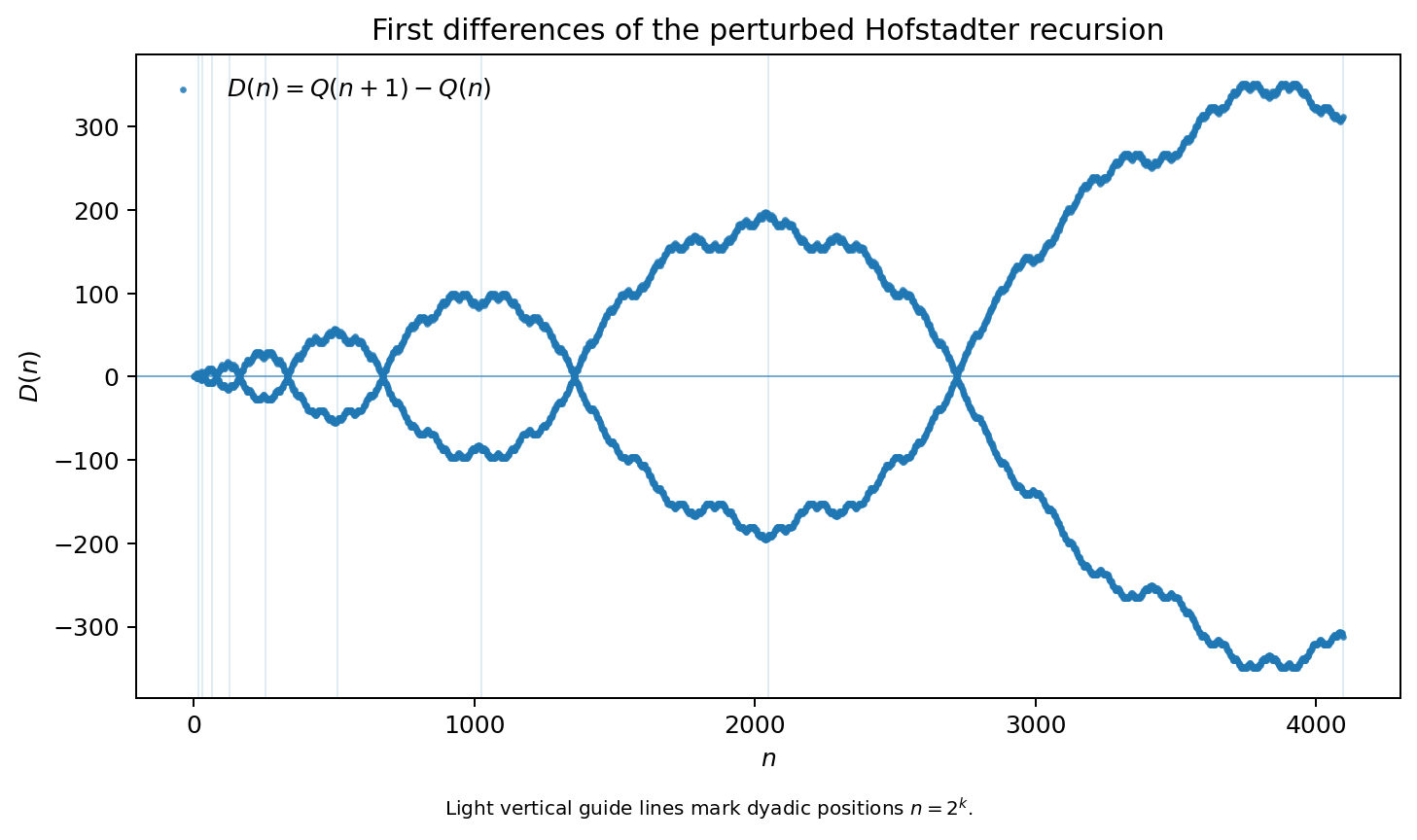}
  \caption{First differences \(\Delta_Q(n)=Q(n+1)-Q(n)\) for \(1\leq n<4096\).
  The odd-indexed points are exactly \(2\delta(m)\); the even-indexed points
  are \(-2\delta(m)\) plus a correction in \(\{0,2\}\), by
  \eqref{eq:first-difference-branches}.  The growing arches are therefore
  rigorous, not a bounded-increment phenomenon.  Light lines at \(n=2^j\)
  are scale guides, not asserted arch boundaries.}
  \label{fig:first-differences}
\end{figure}

\subsection{Exact dyadic scaling and centered-fluctuation profiles}

The self-similarity has an exact combinatorial meaning.  The positive-arch
half-lengths \(a_r\) introduced in Section~\ref{sec:framework} satisfy
\[
 a_{r+1}=4a_r-1,
\]
and the binary words of consecutive arches are related by the two exact-fit
interleavings \eqref{eq:law-one}--\eqref{eq:law-two}.  Passing from one
positive arch to the next is therefore the exact affine dyadic dilation
\(a_r\mapsto4a_r-1\), asymptotically a fourfold change of scale, accompanied
by an explicit word transformation.  This statement does not imply a
pointwise dilation identity such as \(Q(2n)=2Q(n)\).

A complementary finite-scale view comes from the centered fluctuation
function
\[
 E(n)=Q(n)-\frac n2.
\]
The two parity envelopes have exact binary slopes:
\begin{align}
 E(2m+1)-E(2m-1)&=2\Delta U(m)-1\in\{-1,1\},\nonumber\\
 E(2m+2)-E(2m)&=2\Delta V(m)-1\in\{-1,1\}.
 \label{eq:centered-parity-slopes}
\end{align}
Figure~\ref{fig:dyadic-fluctuation} displays nested dyadic subsamples of these
envelopes. For this figure, let \(N=2^{12}\) and let
\(M_k^{\mathrm{fig}}=2^{12-k}-1\) denote the number of sample points in profile \(k\). Set
\[
 \mathcal E_k(m)=E(2^k m)=Q(2^k m)-2^{k-1}m
 \qquad(1\leq m\leq M_k^{\mathrm{fig}}).
\]
Thus profile \(k\) samples all indices divisible by \(2^k\); it is not the
class of exact 2-adic valuation \(k\).  Profile \(k=0\) contains both parity
branches, whereas every \(k\geq1\) lies on the even branch.  The lower-panel
overlap primarily visualizes this parity split and nested decimation; no
dilation law or limiting profile is inferred.

\begin{figure}[t]
  \centering
  \includegraphics[width=.96\textwidth,trim=0 0 0 21pt,clip]
    {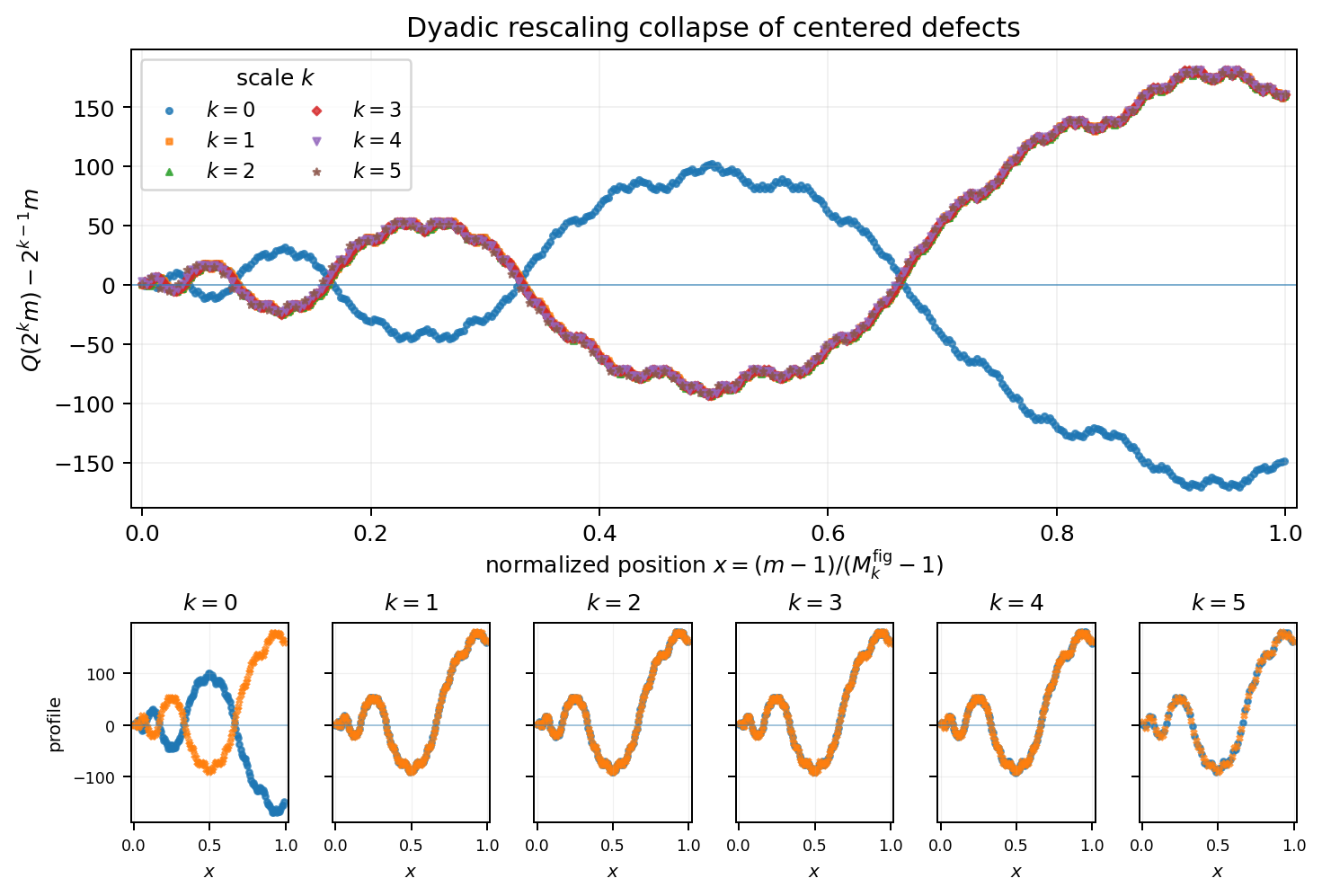}
  \caption{Nested dyadic subsamples of the centered fluctuation function
  for \(n<N=4096\).
  In the upper panel, for \(k=0,\ldots,5\) and \(1\leq m\leq M_k^{\mathrm{fig}}\), the ordinate
  is \(\mathcal E_k(m)\) and the horizontal coordinate is
  \(x=(m-1)/(M_k^{\mathrm{fig}}-1)\).  Lower panel \(k\) overlays the profiles at scales
  \(k\) and \(k+1\) on their normalized grids.  The ordinate is the centered
  fluctuation \(E(n)=Q(n)-n/2\), sampled at \(n=2^k m\).  The visible
  alignment mainly reflects nested decimation; no limiting profile is
  asserted.}
  \label{fig:dyadic-fluctuation}
\end{figure}

\section{Exact binary framework}
\label{sec:framework}

This section fixes notation and records the precise ingredients imported from
Cloitre \cite{cloitre2026mantovanelli}.  The index conventions used below are
collected here to make boundary checks explicit.

\begin{center}
\small
\begin{tabular}{@{}p{.49\textwidth}p{.39\textwidth}@{}}
\toprule
Object & Index range or convention \\
\midrule
letters \(W[t]\) & \(0\leq t<|W|\) \\
slice \(W[a:b]\) & \(a\leq t<b\), zero-based and half-open \\
zero cuts \(p_i\) & \(0\leq i\leq n\) \\
zero gaps \(\omega_i\), degrees \(c_i\) & \(1\leq i\leq n\) \\
reflection & \(\rho(i)=n+1-i\) \\
plateau indices of \(U,V\) & \(m\geq1\) \\
\bottomrule
\end{tabular}
\end{center}
Thus, in particular, \(c_{\rho(i)}=c_{n+1-i}\).

\subsection{The odd--even split and the arches}

Recall the normalized branches \(U,V\) from \eqref{eq:UV}.  Cloitre's
Theorem~8.4 and Lemma~7.1 imply that they are globally defined, unbounded and
nondecreasing, and
\begin{equation}
 \Delta U(m),\Delta V(m)\in\{0,1\}.
 \label{eq:binary-increments}
\end{equation}

For completeness, we recall the word coordinates.  Put
\begin{equation}
 a_0=3,
 \qquad
 a_{r+1}=4a_r-1,
 \qquad
 a_r=\frac{2\cdot4^{r+1}+1}{3},
 \label{eq:a-r}
\end{equation}
and
\begin{equation}
 u_r=2a_r-r-2,
 \qquad
 v_r=4a_r-r-2.
 \label{eq:u-v}
\end{equation}
The word \(\Interleave(X,Y,\varepsilon)\) is produced by a two-state
transducer: in state \(0\) it reads the next letter of \(X\), in state \(1\)
the next letter of \(Y\), emits that letter, and uses the emitted letter as
the new state.  The orbital words are
\begin{align}
 P_0&=001011,\nonumber\\
 N_r&=\Interleave\bigl(P_r[2:|P_r|],P_r[0:|P_r|-1],1\bigr),
 \label{eq:law-one}\\
 P_{r+1}&=\Interleave(00N_r1,0N_r,0).
 \label{eq:law-two}
\end{align}
Cloitre's exact-fit lemma (Lemma~6.6) proves that neither interleaving uses a
fallback.  Moreover,
\begin{equation}
 |P_r|=2a_r,
 \qquad
 |N_r|=4a_r-3,
 \label{eq:word-lengths}
\end{equation}
and the increments on the two alternating arches are
\begin{align}
 \Delta U(u_r+t)&=P_r[t],
 &\Delta V(u_r+t)&=1-P_r[t]
 &&(0\leq t<2a_r),
 \label{eq:positive-arch}\\
 \Delta V(v_r+t)&=N_r[t]
 &&&(0\leq t<4a_r-3),
 \label{eq:negative-v}\\
 \Delta U(v_r)&=1,
 &\Delta U(v_r+t)&=1-N_r[t-1]
 &&(1\leq t<4a_r-3).
 \label{eq:negative-u}
\end{align}
At the endpoints,
\begin{equation}
 U(u_r)=V(u_r)=a_r,
 \qquad
 U(v_r)=V(v_r)=2a_r,
 \qquad
 U(u_{r+1})=V(u_{r+1})=a_{r+1}.
 \label{eq:arch-endpoints}
\end{equation}

\subsection{Cores, gaps and canonical trees}

For a balanced word \(W\) with \(n\) zeros and \(n\) ones, let
\begin{equation}
 0=p_0<p_1<\cdots<p_n
 \label{eq:zero-cuts}
\end{equation}
be the cuts immediately after its zeros.  Its zero gaps and degrees are
\begin{equation}
 \omega_i(W)=p_i-p_{i-1},
 \qquad
 c_i(W)=\omega_i(W)-1
 \quad(1\leq i\leq n).
 \label{eq:gaps-degrees}
\end{equation}
We write
\[
 \Deg(W)=\ms{c_i(W):1\leq i\leq n}.
\]
The reversal-complement of \(W\) is
\(W^*=\overline{\rev(W)}\).  An anti-palindromic word satisfies \(W=W^*\).

For a Dyck word \(W\), define
\begin{equation}
 T(W)=\Interleave(0W1,W,0),
 \qquad
 \Cword(W)=T(W)[1:|T(W)|-1].
 \label{eq:T-C}
\end{equation}
Set
\begin{equation}
 \Acore_r=P_r[1:|P_r|-1],
 \qquad
 \Ccore_r=\Cword(\Acore_r)=N_r[0:|N_r|-1].
 \label{eq:two-cores}
\end{equation}
Cloitre's Lemma~9.4 gives \(T(\Acore_r)=0N_r\) and
\(T^2(\Acore_r)=P_{r+1}\). Removing the outer zero and one gives
\eqref{eq:two-cores} and the second core identity
\begin{equation}
 \Acore_{r+1}=\Cword(0\Ccore_r1).
 \label{eq:wrapped-core}
\end{equation}
Both cores are anti-palindromic Dyck words. The base
\(\Acore_0=0101\) has these properties, \(\Cword\) preserves Dyck words
and commutes with reversal-complementation, and adjoining an outer zero and
one preserves both properties
\cite[Lemmas~9.3 and~10.4]{cloitre2026mantovanelli}.

For the forest interpretation we use Cloitre's canonical forest of
\(\Acore_r\) \cite[Corollary~10.3, Lemma~11.1, and Corollary~12.3]
{cloitre2026mantovanelli}. Vertex \(i\) has \(c_i\) children. The root
degrees form the suffix
\begin{equation}
 (1,3,\ldots,2r+1)
 \label{eq:root-types}
\end{equation}
of the degree word, and a vertex of degree \(L\) has one child of each
degree \(0,\ldots,L-1\). In a forest with \(n\) vertices and \(r+1\)
roots, the degrees sum to \(n-(r+1)\). The degrees count the ones
preceding the successive zeros, so the remaining \(r+1\) ones form the
terminal run of \(\Acore_r\).

For \(L\geq0\), let \(\Tcal_L\) be the canonical rooted tree of type \(L\):
\(\Tcal_0\) is one vertex, while the root of \(\Tcal_L\) has one child of
each type \(0,1,\ldots,L-1\), in the order prescribed by the canonical
permutation.  The order is irrelevant for degree multisets.  Thus
\begin{equation}
 \Deg(\Tcal_L)
 =\ms{L}\mplus\Deg(\Tcal_0)\mplus\cdots\mplus\Deg(\Tcal_{L-1}).
 \label{eq:tree-recursion}
\end{equation}
The first four cases make this recursion concrete:
\[
 \begin{array}{c@{\qquad}l}
 \toprule
 \text{tree}&\text{degree multiset}\\
 \midrule
 \Tcal_0&\ms{0}\\
 \Tcal_1&\ms{1,0}\\
 \Tcal_2&\ms{2,1,0,0}\\
 \Tcal_3&\ms{3,2,1,1,0,0,0,0}\\
 \bottomrule
 \end{array}
\]

For a multiset \(\mathcal M\) and a number \(a\), write
\begin{equation}
 a+\mathcal M:=\ms{a+x:x\in\mathcal M}.
 \label{eq:multiset-shift}
\end{equation}

\subsection{The exact gap transform}

The next local transform will be used both for even half-arches and for the
canonical-tree decomposition.

\begin{lemma}[Exact gap transform]
\label{lem:gap-transform}
Let \(W\) be a Dyck word of semilength \(n\), with zero gaps
\(\omega_1,\ldots,\omega_n\).  If
\(e_1,\ldots,e_{2n}\) is the degree word of \(\Cword(W)\), then, for
\(0\leq t<2n\),
\begin{equation}
 e_{t+1}=
 \begin{cases}
 0,&W[t]=0,\\
 \omega_{O_W(t+1)},&W[t]=1,
 \end{cases}
 \label{eq:degree-transform}
\end{equation}
where \(O_W(t+1)\) is the number of ones in \(W[0:t+1]\).  Consequently,
\begin{equation}
 \Deg(\Cword(W))=\ms{0^{[n]}}\mplus\bigl(1+\Deg(W)\bigr).
 \label{eq:multiset-gap-transform}
\end{equation}
\end{lemma}

\begin{proof}
Cloitre's Lemma~9.2 gives the gaps of \(\Cword(W)\), in the order indexed by
the letters of \(W\): a zero creates gap \(1\), while the \(j\)-th one
creates gap \(1+\omega_j\).  Subtracting one gives
\eqref{eq:degree-transform}.  There are \(n\) letters of each kind, which
gives \eqref{eq:multiset-gap-transform}.
\end{proof}

\begin{corollary}[Final degree under the gap transform]
\label{cor:final-degree-transform}
If the final degree of \(\Acore_r\) is \(2r+1\), then the final degree of
\(\Ccore_r=\Cword(\Acore_r)\) is \(2r+2\).
\end{corollary}

\begin{proof}
A nonempty Dyck word ends in its \(n\)-th one.  Formula
\eqref{eq:degree-transform} therefore makes the last degree of
\(\Cword(W)\) equal to \(\omega_n(W)=c_n(W)+1\).
\end{proof}

\section{Frequencies as folded plateau lengths}
\label{sec:plateaux}

For \(X\in\{U,V\}\), define its first hitting time and its plateau length by
\begin{equation}
 \tau_X(s)=\min\{m\geq1:X(m)=s\},
 \qquad
 L_X(s)=\#\{m\geq1:X(m)=s\}.
 \label{eq:hitting-local}
\end{equation}
They are finite because \(X\) is unbounded and has binary increments.

\begin{lemma}[Plateau identity]
\label{lem:plateau-identity}
For every \(s\geq1\),
\begin{equation}
 F(s)=L_U(s)+L_V(s)
     =\tau_U(s+1)-\tau_U(s)+\tau_V(s+1)-\tau_V(s).
 \label{eq:plateau-identity}
\end{equation}
\end{lemma}

\begin{proof}
By \eqref{eq:UV}, \(Q(2m-1)=2s-1\) exactly when \(U(m)=s\), and
\(Q(2m)=2s-1\) exactly when \(V(m)=s\).  This proves the first equality.
A nondecreasing integer sequence with increments in \(\{0,1\}\) occupies
level \(s\) at the consecutive indices from \(\tau_X(s)\) through
\(\tau_X(s+1)-1\).  This proves the second equality.
\end{proof}

For \(r\geq0\), let
\[
 g^{(r)}=(g_1,\ldots,g_{a_r-1})
       =(\omega_i(\Acore_r))_{i=1}^{a_r-1}
\]
and
\[
 h^{(r)}=(h_1,\ldots,h_{2a_r-2})
       =(\omega_i(\Ccore_r))_{i=1}^{2a_r-2}.
\]

\begin{lemma}[Interior plateau formula]
\label{lem:interior-plateau}
For every \(r\geq0\),
\begin{align}
 F(a_r+j)&=g_j+g_{a_r-j}
 &&(1\leq j\leq a_r-1),
 \label{eq:positive-internal}\\
 F(2a_r+j)&=h_j+h_{2a_r-1-j}
 &&(1\leq j\leq2a_r-2).
 \label{eq:negative-internal}
\end{align}
\end{lemma}

\begin{proof}
Consider first the positive arch.  A zero of \(P_r\) is an increment of
\(V\), and a one is an increment of \(U\).  For
\(1\leq j\leq a_r-1\), the plateau length of \(V\) at \(a_r+j\) is the
spacing between the \(j\)-th and \((j+1)\)-st zero of \(P_r\).  Removing the
first zero of \(P_r\) makes this spacing \(g_j\).  The corresponding plateau
length of \(U\) is the spacing between consecutive ones.  Since
\(\Acore_r\) is anti-palindromic, its one spacings are its zero spacings in
reverse order.  Hence that length is \(g_{a_r-j}\), proving
\eqref{eq:positive-internal}.

The same indexing works on the negative arch after the one-step shift in
\eqref{eq:negative-u}.  Deleting the final one of \(N_r\) gives the balanced
anti-palindromic core \(\Ccore_r\).  Here the shifted complement in
\eqref{eq:negative-u} makes its zero spacings the \(U\)-plateau lengths,
while the one spacings are the \(V\)-plateau lengths.  Anti-palindromicity
reverses the latter list, proving \eqref{eq:negative-internal}.
\end{proof}

\begin{lemma}[Boundary plateau formula]
\label{lem:boundary-plateau}
For every \(r\geq0\), the two boundary levels have the exact local times
\begin{equation}
 \begin{array}{c|ccc}
 \text{level}&L_U&L_V&F\\
 \hline
 a_r&2r+3&1&2r+4\\
 2a_r&1&2r+4&2r+5
 \end{array}
 \label{eq:boundary-table}
\end{equation}
and the terminal gaps are
\begin{equation}
 g_{a_r-1}=2r+2,
 \qquad
 h_{2a_r-2}=2r+3.
 \label{eq:terminal-gaps}
\end{equation}
Equivalently,
\begin{equation}
 F(a_r)=2+g_{a_r-1},
 \qquad
 F(2a_r)=2+h_{2a_r-2}.
 \label{eq:boundary-gap-formulas}
\end{equation}
\end{lemma}

\begin{proof}
The canonical forest of \(\Acore_r\) has \(r+1\) roots, so the terminal one
run of \(\Acore_r\) has length \(r+1\).  Since \(\Acore_r\) is
anti-palindromic, its initial zero run has the same length.  Consequently,
\(P_r=0\Acore_r1\) has extreme runs of length \(r+2\).  To locate the first
run of \(\Ccore_r\), apply the degree
transform \eqref{eq:degree-transform}: the first \(r+1\) letters of
\(\Acore_r\) are zeros, so the first \(r+1\) degrees of \(\Ccore_r\) are
zero; the next letter is the first one and produces degree
\(\omega_1(\Acore_r)=1\).  Thus \(\Ccore_r\) has initial zero run exactly
\(r+1\).  Anti-palindromicity gives the same terminal one-run length.  Hence
\(N_r\) has an initial zero run of length \(r+1\) and a terminal one run of
length \(r+2\).

For \(r=0\), the displayed initial values of the normalized branches give
explicitly
\[
 (L_U(3),L_V(3))=(3,1),
 \qquad
 (L_U(6),L_V(6))=(1,4),
\]
which is \eqref{eq:boundary-table}.  Now let \(r\geq1\).  At \(u_r\), the
\(V\)-plateau at \(a_r\) consists of one index.  Equations
\eqref{eq:negative-v}--\eqref{eq:negative-u} use all
but the last letter of the terminal one run of \(N_{r-1}\); this gives the
\(r\) preceding zero increments of \(U\).  The initial zero run of \(P_r\)
gives \(r+2\) following zero increments.  Hence the \(U\)-plateau has
\(r\) indices before \(u_r\), the index \(u_r\) itself, and \(r+2\) further
indices, so its length is \(2r+3\).  Thus
\(F(a_r)=2r+4\).

At \(v_r\), the final one run of \(P_r\) gives \(r+2\) preceding zero
increments of \(V\), and the initial zero run of \(N_r\) gives \(r+1\)
following zero increments.  Thus the \(U\)-plateau at \(2a_r\) consists of
one index, while the \(V\)-plateau runs from \(v_r-(r+2)\) through
\(v_r+(r+1)\), and therefore has length \(2r+4\).  Hence
\(F(2a_r)=2r+5\).

It remains to identify the terminal gaps.  The last root in
\eqref{eq:root-types} has degree \(2r+1\), hence
\(g_{a_r-1}=2r+2\).  Corollary~\ref{cor:final-degree-transform} makes the
final degree of \(\Ccore_r\) equal to \(2r+2\), hence
\(h_{2a_r-2}=2r+3\).  This proves \eqref{eq:terminal-gaps} and
\eqref{eq:boundary-gap-formulas}.
\end{proof}

Put
\begin{equation}
 H(s)=F(s)-3.
 \label{eq:H}
\end{equation}
Lemmas~\ref{lem:interior-plateau} and \ref{lem:boundary-plateau} express
\(H\) along every half-arch in ordered gap coordinates.

The two types of half-arches can now be indexed uniformly.  For \(r\geq0\),
set
\begin{align}
 W_{2r+1}&=\Acore_r,
 &s_{2r+1}&=a_r,
 &n_{2r+1}&=a_r-1,
 \label{eq:odd-W}\\
 W_{2r+2}&=\Ccore_r,
 &s_{2r+2}&=2a_r,
 &n_{2r+2}&=2a_r-2.
 \label{eq:even-W}
\end{align}
Thus \(W_k\) has semilength \(n_k\).  If
\(c^{(k)}=(c_1,\ldots,c_{n_k})\) is its degree word, then
Lemmas~\ref{lem:interior-plateau} and \ref{lem:boundary-plateau} become
\begin{align}
 H(s_k)&=k,
 \label{eq:uniform-start}\\
 H(s_k+j)&=c_j+c_{n_k+1-j}-1
 &&(1\leq j\leq n_k).
 \label{eq:uniform-fold}
\end{align}
The elementary length identities needed for the dyadic cut are
\begin{equation}
 2^k=s_{k-1}+\frac{n_{k-1}}2,
 \qquad
 2^{k+1}=s_k+\frac{n_k}2
 \qquad(k\geq2).
 \label{eq:dyadic-midpoints}
\end{equation}
They follow at once from \eqref{eq:a-r}; importantly, the dyadic boundaries
are \emph{midpoints} of two consecutive half-arches.

\section{Reflection pairs and canonical tree degrees}
\label{sec:reflection}

\begin{lemma}[Tree and half-arch degree decompositions]
\label{lem:tree-decomposition}
For \(L\geq0\), the canonical tree satisfies
\begin{equation}
 |\Tcal_L|=2^L,
 \qquad
 \#\{v\in\Tcal_L:\deg(v)>0\}=
 \begin{cases}
 0,&L=0,\\
 2^{L-1},&L\geq1.
 \end{cases}
 \label{eq:tree-size-positive}
\end{equation}
Moreover,
\begin{align}
 \Deg(W_{2r+1})
 &=\Deg(\Tcal_1)\mplus\Deg(\Tcal_3)\mplus\cdots
   \mplus\Deg(\Tcal_{2r+1}),
 \label{eq:odd-tree-decomp}\\
 \Deg(W_{2r+2})
 &=\Deg(\Tcal_2)\mplus\Deg(\Tcal_4)\mplus\cdots
   \mplus\Deg(\Tcal_{2r+2}).
 \label{eq:even-tree-decomp}
\end{align}
\end{lemma}

\begin{proof}
The recursion \eqref{eq:tree-recursion} gives
\(|\Tcal_0|=1\) and
\(|\Tcal_L|=1+\sum_{j<L}|\Tcal_j|=2^L\).  It also gives one positive root
plus all positive vertices in the children.  Induction yields
\(1+\sum_{j=1}^{L-1}2^{j-1}=2^{L-1}\) for \(L\geq1\).

The root word \eqref{eq:root-types} and the canonical forest grammar give
\eqref{eq:odd-tree-decomp}.  To obtain the even decomposition, note that
\begin{equation}
 \ms{0^{[|\Tcal_L|]}}\mplus\bigl(1+\Deg(\Tcal_L)\bigr)
 =\Deg(\Tcal_{L+1}).
 \label{eq:tree-suspension-multiset}
\end{equation}
Indeed, partition the \(|\Tcal_L|=1+\sum_{j<L}|\Tcal_j|\) zeros into one
singleton and blocks of sizes \(|\Tcal_j|\), and use induction together with
\eqref{eq:tree-recursion}.  Applying
\eqref{eq:multiset-gap-transform} to every tree component in
\eqref{eq:odd-tree-decomp} proves \eqref{eq:even-tree-decomp}.
\end{proof}

\begin{lemma}[Paired degrees and the central pair]
\label{lem:paired-degrees}
For every \(k\geq1\), the semilength \(n_k\) of \(W_k\) is even.
Write \(c_i=c_i(W_k)\). Exactly one of \(c_i,c_{n_k+1-i}\) is positive
for each \(1\leq i\leq n_k/2\), and
\begin{equation}
 (c_{n_k/2},c_{n_k/2+1})=
 \begin{cases}
 (0,1),&k\text{ odd},\\
 (1,0),&k\text{ even}.
 \end{cases}
 \label{eq:central-pair}
\end{equation}
\end{lemma}

\begin{proof}
Let \(Y\) be an anti-palindromic Dyck word of length \(2m\).
The gap transform \eqref{eq:degree-transform} assigns a positive degree
in \(\Cword(Y)\) exactly to the positions occupied by ones in \(Y\).
Reflected letters of \(Y\) are complementary, so exactly one degree in
each reflected pair of \(\Cword(Y)\) is positive. This applies to
\(\Ccore_r=\Cword(\Acore_r)\) and to
\(\Acore_{r+1}=\Cword(0\Ccore_r1)\) by \eqref{eq:wrapped-core}.
The remaining core \(\Acore_0=0101\) has degree word \((0,1)\).
Every image has even semilength \(2m\), and the base has semilength two.

To locate the positive degree of the central pair, take
\(X=\Acore_r\) or \(X=\Ccore_r\), and write \(n=|X|/2\) and
\(p=n/2\). By the pair property, \(p\) degrees of \(X\) are positive.
Each of them counts a nonempty run of ones before a zero. Including the
terminal run gives \(p+1\) runs of ones and as many runs of zeros.
Anti-palindromicity makes the run lengths palindromic, so the first half
of \(X\) comprises its first \(p+1\) runs. Since the first run consists
of zeros, the central factor \(X[n-1]X[n]\) is \(10\) if \(p\) is odd
and \(01\) if \(p\) is even.

Suppose the first half contains \(s\) ones and \(i=n-s\) zeros. It
contains at least two runs, so \(s>0\). There are thus zeros in the
second half. The first such zero is preceded by a one in either central
configuration. It is the \((i+1)\)-st zero of \(X\), which gives
\(c_{i+1}(X)>0\). Its reflected index is \(n+1-(i+1)=s\), whence
\begin{equation}
 \omega_s(X)=1.
 \label{eq:central-source-gap}
\end{equation}

For \(X=\Acore_r\) we have \(p=(a_r-1)/2=1+4+\cdots+4^r\), which is
odd. The central \(1\) in the factor \(10\) is the \(s\)-th one of \(X\).
Equations~\eqref{eq:degree-transform} and \eqref{eq:central-source-gap}
therefore give \((1,0)\) at the center of the degree word of
\(\Ccore_r=\Cword(X)\).

For \(X=\Ccore_r\) we have \(p=a_r-1\), which is even. The central
factor is \(01\), and adjoining an outer zero and one leaves it at the
center of \(Y=0X1\). Its one is the \((s+1)\)-st one of \(Y\).
The zero gaps of \(Y\) are \(1,\omega_1(X),\ldots,\omega_n(X)\), so
\(\omega_{s+1}(Y)=\omega_s(X)=1\). Applying the gap transform to \(Y\)
gives \((0,1)\) at the center of the degree word of
\(\Acore_{r+1}=\Cword(Y)\). The degree word of \(\Acore_0\) gives the
same pair in the base case.
\end{proof}

For a word satisfying Lemma~\ref{lem:paired-degrees}, define its shifted
positive-degree multiset by
\begin{equation}
 \Sh(W)=\ms{c_i-1:c_i>0}.
 \label{eq:Sh}
\end{equation}
There is one element of \(\Sh(W)\) for each reflection pair.

\begin{lemma}[Shifted tree degrees]
\label{lem:shifted-tree}
For every \(L\geq1\),
\begin{equation}
 \ms{d-1:d\in\Deg(\Tcal_L),\ d>0}=\Deg(\Tcal_{L-1}).
 \label{eq:shifted-tree}
\end{equation}
Consequently,
\begin{align}
 \Sh(W_{2r+1})
 &=\Deg(\Tcal_0)\mplus\Deg(\Tcal_2)\mplus\cdots
   \mplus\Deg(\Tcal_{2r}),
 \label{eq:Sh-odd}\\
 \Sh(W_{2r+2})
 &=\Deg(\Tcal_1)\mplus\Deg(\Tcal_3)\mplus\cdots
   \mplus\Deg(\Tcal_{2r+1}).
 \label{eq:Sh-even}
\end{align}
\end{lemma}

\begin{proof}
Equation~\eqref{eq:tree-suspension-multiset}, with \(L-1\) in place of
\(L\), says that the positive degrees of \(\Tcal_L\) are exactly
\(1+\Deg(\Tcal_{L-1})\).  Subtracting one proves
\eqref{eq:shifted-tree}.  Applying this identity componentwise to
\eqref{eq:odd-tree-decomp} and \eqref{eq:even-tree-decomp} gives
\eqref{eq:Sh-odd} and \eqref{eq:Sh-even}.
\end{proof}

\section{The dyadic frequency theorem}
\label{sec:frequency-theorem}

We can now perform the exact fold.  Let
\begin{equation}
 \Bcal_k=\{s:2^k\leq s<2^{k+1}\}.
 \label{eq:block}
\end{equation}

\begin{proposition}[Folded half-arch correspondence]
\label{prop:folded-correspondence}
For every \(k\geq1\),
\begin{equation}
 \ms{H(s):s\in\Bcal_k}=\Deg(\Tcal_k).
 \label{eq:H-tree}
\end{equation}
\end{proposition}

\begin{proof}
The case \(k=1\) is direct:
\((F(2),F(3))=(3,4)\), so the excess multiset is
\(\ms{0,1}=\Deg(\Tcal_1)\).

Let \(k\geq2\).  By \eqref{eq:dyadic-midpoints}, the block begins halfway
through \(W_{k-1}\), passes through the new half-arch start \(s_k\), and ends
just before the midpoint of \(W_k\).  The definitions
\eqref{eq:odd-W}--\eqref{eq:even-W} also give the adjacency identity
\begin{equation}
 s_k=s_{k-1}+n_{k-1}+1,
 \label{eq:half-arch-adjacency}
\end{equation}
so no level is omitted between the two halves.  More precisely, the block
consists of
\begin{equation}
 \left\{s_{k-1}+j:\frac{n_{k-1}}2\leq j\leq n_{k-1}\right\},
 \quad \{s_k\},
 \quad
 \left\{s_k+j:1\leq j<\frac{n_k}2\right\}.
 \label{eq:block-split}
\end{equation}
The interval geometry is therefore
\[
 \underbrace{s_{k-1},\ldots,2^k-1}_{\text{discarded}}
 \ \big|\ 
 \underbrace{2^k,\ldots,s_k-1}_{\text{tail of }W_{k-1}}
 \ \big|\ \boxed{s_k}\ \big|\ 
 \underbrace{s_k+1,\ldots,2^{k+1}-1}_{\text{head of }W_k}
 \ \big|\ \cdots .
\]
Its cardinality check is
\begin{equation}
 \left(\frac{n_{k-1}}2+1\right)+1+
 \left(\frac{n_k}2-1\right)=2^k.
 \label{eq:block-cardinality}
\end{equation}

In the first set of \eqref{eq:block-split}, formula
\eqref{eq:uniform-fold} takes one value \(d-1\) from every reflected degree
pair.  More explicitly, the two indices
\(j=n_{k-1}/2\) and \(j=n_{k-1}/2+1\) generate the same central reflected
pair.  By \eqref{eq:central-pair}, its degrees sum to one, so both folded
values are \(1-1=0\).  Thus the central-pair value is taken twice, and the
multiset of the first interval is
\begin{equation}
 \Sh(W_{k-1})\mplus\ms{0}.
 \label{eq:old-tail}
\end{equation}
The last set of \eqref{eq:block-split} takes one value from every pair of
\(W_k\) except the central pair.  Its multiset is therefore \(\Sh(W_k)\)
with the central shifted contribution \(1-1=0\) removed.  Finally,
\eqref{eq:uniform-start} contributes \(k\).  The extra central zero from the
old tail and the omitted central zero from the new head cancel exactly, and
\begin{equation}
 \ms{H(s):s\in\Bcal_k}
 =\Sh(W_{k-1})\mplus\ms{k}\mplus\Sh(W_k).
 \label{eq:block-Sh}
\end{equation}

Equations~\eqref{eq:Sh-odd} and \eqref{eq:Sh-even} show that the two
\(\Sh\)-terms in \eqref{eq:block-Sh} contain, once each, the degree
multisets of all tree types \(0,1,\ldots,k-1\).  Adding the singleton root
degree \(k\) gives \(\Deg(\Tcal_k)\) by
\eqref{eq:tree-recursion}.  This proves \eqref{eq:H-tree}.
\end{proof}

\begin{lemma}[Degree distribution of \(\Tcal_k\)]
\label{lem:tree-degree-law}
For \(k\geq0\), degree \(h\) occurs in \(\Tcal_k\) with multiplicity
\begin{equation}
 \#\{v\in\Tcal_k:\deg(v)=h\}=
 \begin{cases}
 2^{k-h-1},&0\leq h<k,\\
 1,&h=k.
 \end{cases}
 \label{eq:degree-counts}
\end{equation}
Equivalently,
\begin{equation}
 \Deg(\Tcal_k)=\ms{\ordtwo(j):1\leq j\leq2^k}.
 \label{eq:tree-ruler}
\end{equation}
\end{lemma}

\begin{proof}
The assertion is immediate for \(k=0\).  In
\eqref{eq:tree-recursion}, the root contributes one degree \(k\).  For
\(h<k\), the number of degree-\(h\) vertices is the sum of that number over
\(\Tcal_0,\ldots,\Tcal_{k-1}\).  Induction gives
\[
 1+\sum_{j=h+1}^{k-1}2^{j-h-1}=2^{k-h-1}.
\]
Among \(1,\ldots,2^k\), exactly \(2^{k-h-1}\) integers have 2-adic
valuation \(h<k\), and only \(2^k\) has valuation \(k\).  This proves
\eqref{eq:tree-ruler}.
\end{proof}

\begin{proof}[Proof of Theorem~\ref{thm:intro-frequency}]
For \(k=0\), the value \(1\) occurs three times, so both sides are
\(\ms{3}\).  For \(k\geq1\), combine
Proposition~\ref{prop:folded-correspondence},
Lemma~\ref{lem:tree-degree-law}, and \(F=H+3\).
\end{proof}

\begin{example}[The first nontrivial fold]
\label{ex:first-fold}
For \(k=2\), the relevant half-arch cores have degree words
\[
 c^{(1)}=(0,1),
 \qquad
 c^{(2)}=(0,1,0,2).
\]
The block \(\Bcal_2=\{4,5,6,7\}\) contains the reflected tail of \(W_1\),
the new half-arch start, and the first reflected pair of \(W_2\).  Formula
\eqref{eq:uniform-fold} therefore gives
\[
 (H(4),H(5),H(6),H(7))=(0,0,2,1),
\]
whose multiset is \(\Deg(\Tcal_2)=\ms{2,1,0,0}\), although the displayed
order is not canonical.  Adding three gives
\((F(4),F(5),F(6),F(7))=(3,3,5,4)\).
\end{example}

\begin{example}[A complete second fold]
\label{ex:second-fold}
For \(k=3\), the distinction between an ordered frequency sequence and the
multiset assertion is visible in full:
\[
\begin{array}{@{}l@{\quad}l@{}}
\toprule
\text{object}&\text{data}\\
\midrule
\text{actual frequency order on }8\leq s<16
 &(3,3,4,6,5,3,3,4)\\
\text{sorted frequency multiset}
 &\ms{3,3,3,3,4,4,5,6}\\
\text{actual excess order}
 &(0,0,1,3,2,0,0,1)\\
\text{canonical tree degree multiset }\Deg(\Tcal_3)
 &\ms{3,2,1,1,0,0,0,0}\\
\bottomrule
\end{array}
\]
Thus the theorem identifies the two multisets after adding three; it does not
identify their displayed orders.
\end{example}

\begin{corollary}[Exact multiplicities and block mass]
\label{cor:block-consequences}
For \(k\geq0\),
\begin{equation}
 \#\{s\in\Bcal_k:F(s)=3+h\}=
 \begin{cases}
 2^{k-h-1},&0\leq h<k,\\
 1,&h=k,
 \end{cases}
 \label{eq:frequency-counts}
\end{equation}
and
\begin{equation}
 \sum_{s=2^k}^{2^{k+1}-1}F(s)=4\cdot2^k-1.
 \label{eq:block-sum}
\end{equation}
If \(\mathcal N(M)=\sum_{s=1}^M F(s)\), then
\begin{equation}
 \mathcal N(2^K-1)=4(2^K-1)-K
 \qquad(K\geq0).
 \label{eq:cumulative-dyadic}
\end{equation}
\end{corollary}

\begin{proof}
Equation~\eqref{eq:frequency-counts} is
\eqref{eq:degree-counts} shifted by three.  A tree with \(2^k\) vertices has
\(2^k-1\) edges, so its degree sum is \(2^k-1\).  Proposition
\ref{prop:folded-correspondence} therefore gives
\[
 \sum_{s\in\Bcal_k}F(s)=3\cdot2^k+(2^k-1),
\]
which is \eqref{eq:block-sum}.  Summing the complete blocks
\(\Bcal_0,\ldots,\Bcal_{K-1}\) proves
\eqref{eq:cumulative-dyadic}.
\end{proof}

The plateau identity also gives a clean hitting-time consequence without a
separate endpoint calculation.

\begin{corollary}[Combined dyadic hitting time]
\label{cor:hitting-time}
Let \(\Theta(s)=\tau_U(s)+\tau_V(s)\).  Then
\begin{equation}
 \Theta(s+1)-\Theta(s)=F(s),
 \qquad
 \Theta(2^K)=2^{K+2}-K-2.
 \label{eq:hitting-dyadic}
\end{equation}
\end{corollary}

\begin{proof}
The first identity is Lemma~\ref{lem:plateau-identity}.  Since
\(\Theta(1)=2\), equation~\eqref{eq:cumulative-dyadic} gives
\[
 \Theta(2^K)=2+\mathcal N(2^K-1)=2^{K+2}-K-2.
\]
\end{proof}

\begin{corollary}[Locations of the two largest frequencies]
\label{cor:block-peaks}
For \(k\geq1\), the unique maximum of \(F\) on \(\Bcal_k\) is \(k+3\),
attained at
\begin{equation}
 s_k=\left\lceil\frac{2^{k+2}}{3}\right\rceil .
 \label{eq:maximum-location}
\end{equation}
For \(k\geq2\), the second-largest frequency \(k+2\) is attained uniquely
at \(s_k+1\).  On \(\Bcal_1\) the maximum is \(F(3)=4\) and the
second-largest frequency is \(F(2)=3\).  The singleton \(\Bcal_0\)
contains only \(F(1)=3\).
\end{corollary}

\begin{proof}
Substituting \eqref{eq:a-r} in \eqref{eq:odd-W} and \eqref{eq:even-W}
gives \eqref{eq:maximum-location} in both parities.  Thus \(s_k\in\Bcal_k\),
and \eqref{eq:uniform-start} gives \(F(s_k)=k+3\).
For \(k\geq2\), we have \(n_k\geq4\), so
\(s_k+1<s_k+n_k/2=2^{k+1}\) by \eqref{eq:dyadic-midpoints}.
The core \(W_k\) begins with a zero, giving \(c_1=0\), and its terminal
degree is \(c_{n_k}=k\) by \eqref{eq:terminal-gaps}.
Taking \(j=1\) in \eqref{eq:uniform-fold} now gives \(F(s_k+1)=k+2\).
Equation~\eqref{eq:frequency-counts} gives multiplicity one for each of
these two values.  The initial values give the stated cases \(k=0,1\).
\end{proof}

\begin{remark}[The order within a block]
Formula~\eqref{eq:uniform-fold} recovers the ordered frequencies from the
gap word of each core.  A further question is whether the remaining
positions in a dyadic block can be described directly from the level \(s\),
without constructing the cores.  Such a description could also give a
direct formula for \(\mathcal N(M)\) at arbitrary endpoints.
\end{remark}

\begin{figure}[t]
  \centering
  \includegraphics[width=.85\textwidth]{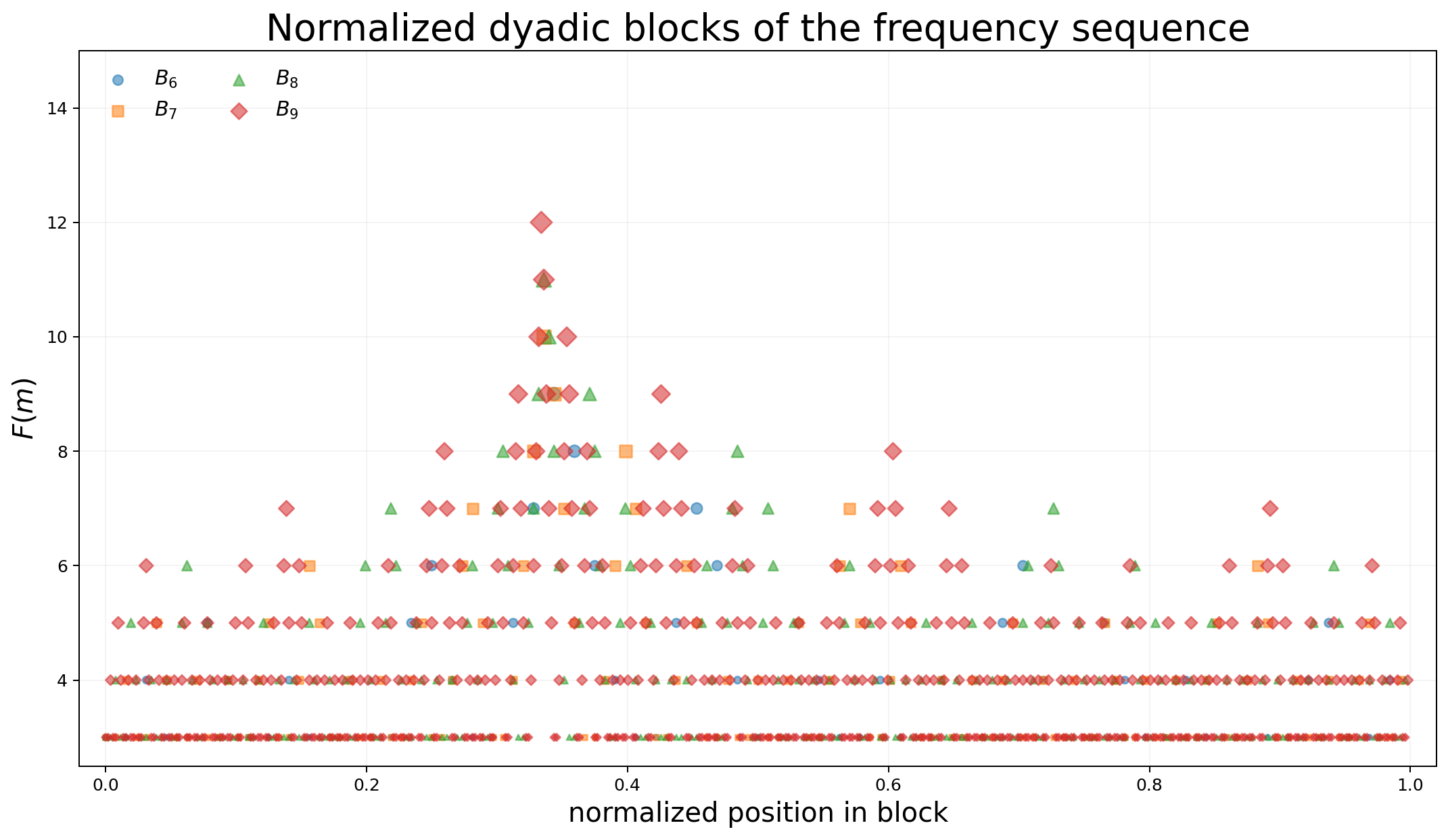}
  \caption{Frequency profiles for \(\Bcal_6,\ldots,\Bcal_9\) on a common
  scale. The maximum lies at \(s_k=\lceil 2^{k+2}/3\rceil\), about one
  third across the block, and the next highest value at \(s_k+1\)
  (Corollary~\ref{cor:block-peaks}).}
  \label{fig:frequency-blocks}
\end{figure}

Let \(G_k(z)=\sum_{m\in\Bcal_k}z^{F(m)-3}\), with \(z^0=1\). For
\(\operatorname{Re}u>1\), write \(\zeta(u)=\sum_{n\geq1}n^{-u}\)
and \(\eta(u)=(1-2^{1-u})\zeta(u)\). The weight \(2^{-ku}\) is constant
on \(\Bcal_k\). Since \(G_k(1)=2^k\), multiplication by \(\eta(u)\)
normalizes the block sum to \(\zeta(u)\) at \(z=1\).

\begin{corollary}[The marked Euler factor at \(2\)]
\label{cor:marked-euler}
For \(\sigma=\operatorname{Re}u>1\) and \(|z|<2^\sigma\),
\begin{equation}
 \begin{aligned}
 \eta(u)\sum_{k\geq0}2^{-ku}G_k(z)
 &=\sum_{n\geq1}\frac{z^{\ordtwo(n)}}{n^u}\\
 &=\frac{1-2^{-u}}{1-z2^{-u}}\,\zeta(u)
   =\frac{1}{1-z2^{-u}}\prod_{p>2}\frac{1}{1-p^{-u}}.
 \end{aligned}
 \label{eq:marked-euler}
\end{equation}
Here \(p\) ranges over the odd primes, and the series and product converge
absolutely.
\end{corollary}

\begin{proof}
By Theorem~\ref{thm:intro-frequency},
\(G_k(z)=\sum_{j=1}^{2^k}z^{\ordtwo(j)}\). The \(2^k\) odd indices in
\(G_{k+1}\) each contribute one, and \(\ordtwo(2i)=1+\ordtwo(i)\) at
the even indices. Thus \(G_{k+1}(z)=2^k+zG_k(z)\) and \(G_0(z)=1\), giving
\[
 \sum_{k\geq0}G_k(z)x^k
 =\frac{1}{1-zx}\left(1+\frac{x}{1-2x}\right)
 =\frac{1-x}{(1-2x)(1-zx)}.
\]
At \(x=2^{-u}\) the geometric expansions converge absolutely since
\(2|x|<1\) and \(|zx|<1\). Multiplication by \(\eta(u)\) gives the
zeta quotient in \eqref{eq:marked-euler}.

Writing \(n=2^r d\) with \(d\) odd bounds the absolute sum of the
Dirichlet series by \(\zeta(\sigma)/(1-|z|2^{-\sigma})\).
Summing over \(r\) gives the factor \((1-z2^{-u})^{-1}\).
The remaining sum over odd \(d\) equals \((1-2^{-u})\zeta(u)\), and
unique factorization gives its Euler product over \(p>2\).
\end{proof}

The same block masses \eqref{eq:block-sum} enter the analysis of the orbit
series \(\sum_{n\ge1}n^{-s}Q(n)^{-t}\). Its decomposition retains the ordered correction
\(C_F(w)=\sum_{m\ge1}(F(m)-4)m^{-w}\), \(w=s+t\), and the occurrence
positions \cite[Theorem~6.2, equations~(37)--(40)]{mantovanelli2026zeta}.

\section{The optimal order of the dyadic defect}
\label{sec:clock}

Define the dyadic defect and its normalized block maxima by
\begin{equation}
 R(n)=Q(2n)-2Q(n),\qquad
 M_k=\max_{2^k\leq n<2^{k+1}}
 \frac{|R(n)|\sqrt{\log_2 n}}{n}.
 \label{eq:Mk}
\end{equation}

\begin{proposition}[Optimal order of the dyadic defect]
\label{prop:defect-order}
The normalized block maxima satisfy
\begin{equation}
 0<\limsup_{k\to\infty} M_k<\infty .
 \label{eq:defect-limsup}
\end{equation}
\end{proposition}

\begin{proof}
For \(E(n)=Q(n)-n/2\), Cloitre's estimate
\cite[Theorem~2.1(iii)]{cloitre2026mantovanelli} gives
\begin{equation}
 R(n)=E(2n)-2E(n)=O\!\left(\frac{n}{\sqrt{\log n}}\right).
 \label{eq:defect-bound}
\end{equation}
This gives the upper bound. For the lower bound, write
\(\delta(m)=V(m)-U(m)\). By \eqref{eq:UV}, adjacent defects satisfy
\[
 R(2m)-R(2m-1)
 =Q(4m)-Q(4m-2)-2\bigl(Q(2m)-Q(2m-1)\bigr)
 =2\Delta V(2m-1)-4\delta(m).
\]
Since \(\Delta V(2m-1)\in\{0,1\}\),
\[
 \max\{|R(2m)|,|R(2m-1)|\}\geq 2|\delta(m)|-1.
\]
Choose \(m_r\) at a maximum of the \(r\)-th positive arch.  Then
\(\delta(m_r)=\mathcal H_r\asymp 4^r/\sqrt r\) by
\eqref{eq:arch-amplitude-overview}, while
\(u_r\leq m_r\leq v_r\) and \eqref{eq:a-r}--\eqref{eq:u-v} give
\(m_r\asymp 4^r\).  For one of \(n_r=2m_r-1,2m_r\) we therefore have
\(|R(n_r)|\geq 2\mathcal H_r-1\).  As \(\log_2 n_r\asymp r\), the
normalized defect at these indices is bounded below by a positive constant
for all sufficiently large \(r\).  The dyadic blocks containing these
indices give the lower bound in \eqref{eq:defect-limsup}.
\end{proof}

Table~\ref{tab:defect} reports \(M_k\) for selected even \(k\), computed by
one forward evaluation of the integer recursion through index \(2^{22}\).

\begin{table}[ht]
\centering
\begin{tabular}{rrrr}
\toprule
\(k\) & \(M_k\) & maximizing \(n\) & \(R(n)\)\\
\midrule
12 & 0.5911485 & 7,830     & 1,287\\
14 & 0.5854757 & 31,584    & 4,783\\
16 & 0.5803729 & 118,256   & 16,719\\
18 & 0.5781764 & 475,720   & 63,335\\
20 & 0.5761864 & 1,912,446 & 241,225\\
\bottomrule
\end{tabular}
\caption{Selected normalized dyadic maxima of the defect.}
\label{tab:defect}
\end{table}

There is also an algebraic description of \(R(n)\). Define the clock by
\begin{equation}
 t(n)=n-Q(n-1)\qquad(n\geq2).
 \label{eq:clock}
\end{equation}

\begin{proposition}[Clock identities]
\label{prop:clock}
For \(m\geq1\),
\begin{equation}
 Q(m)=m+1-t(m+1),\qquad R(m)=2t(m+1)-t(2m+1)-1.
 \label{eq:clock-reconstruct}
\end{equation}
In particular, \(R(m)\) is odd. For \(n\geq3\), the clock satisfies
\begin{align}
 t(n+1)={}&n-2-t(n)-t(n-1)
 \nonumber\\
 &+t\bigl(t(n)+1\bigr)+t\bigl(t(n-1)+2\bigr)-(-1)^n.
 \label{eq:clock-recurrence}
\end{align}
\end{proposition}

\begin{proof}
Definition~\eqref{eq:clock} gives the formula for \(Q(m)\).
Substitution in \(Q(2m)-2Q(m)\) gives the formula for \(R(m)\),
which is odd because \(Q(2m)\) is odd. The two arguments in
\eqref{eq:recurrence} are \(t(n)\) and \(t(n-1)+1\). Expressing \(Q\)
at both arguments by \eqref{eq:clock-reconstruct} and substituting in
\(t(n+1)=n+1-Q(n)\) gives \eqref{eq:clock-recurrence}.
\end{proof}

\section{Alternative initial triples}
\label{sec:seeds}

When three initial values are prescribed, the recurrence must be understood
as starting at \(n=4\).  Under that convention there is a simple rigid orbit.

\begin{proposition}[A noncanonical rigid orbit]
\label{prop:rigid-seed}
Prescribe
\(Q(1)=1,Q(2)=1,Q(3)=2\), and apply \eqref{eq:recurrence} for \(n\geq4\).
Then, for every \(m\geq2\),
\begin{equation}
 Q(2m)=2m,
 \qquad
 Q(2m+1)=2.
 \label{eq:rigid-orbit}
\end{equation}
\end{proposition}

\begin{proof}
Directly,
\[
 Q(4)=Q(2)+Q(3)+1=4,
 \qquad
 Q(5)=Q(1)+Q(3)-1=2.
\]
Assume the asserted formulas through \(2m+1\).  Then
\begin{align*}
 Q(2m+2)
 &=Q\bigl(2m+2-Q(2m+1)\bigr)
   +Q\bigl(2m+2-Q(2m)\bigr)+1\\
 &=Q(2m)+Q(2)+1=2m+2,
\end{align*}
and
\begin{align*}
 Q(2m+3)
 &=Q\bigl(2m+3-Q(2m+2)\bigr)
   +Q\bigl(2m+3-Q(2m+1)\bigr)-1\\
 &=Q(1)+Q(2m+1)-1=2.
\end{align*}
All recursive arguments displayed here are positive and earlier than the
current index, so the induction is legal.
\end{proof}

For context only, Table~\ref{tab:seeds} records a finite experiment over
small triples.  ``Defined'' means that all recursive arguments remained legal
through \(N=10^6\); it is not a global-existence claim.

\begin{table}[ht]
\centering
\begin{tabular}{lll}
\toprule
Initial triple & status through \(10^6\) & observed relation\\
\midrule
\((1,1,1)\) & defined & canonical orbit\\
\((2,1,1),(3,1,1)\) & defined & canonical values from index \(2\) on\\
\((1,3,3)\) & defined & two-index shift of the canonical data\\
\((1,1,2)\) & defined & rigid orbit, Proposition~\ref{prop:rigid-seed}\\
\((1,3,1),(2,1,2)\) & defined & no simple relation detected\\
\((2,2,1),(2,3,1),(3,1,2)\) & defined & no simple relation detected\\
\bottomrule
\end{tabular}
\caption{Selected initial triples, with the recurrence started at \(n=4\).}
\label{tab:seeds}
\end{table}

The table is deliberately not called a classification.  Proving global
existence or eventual equivalence for noncanonical triples requires separate
arguments.

\section{Conclusion and open problems}

Theorem~\ref{thm:intro-frequency} identifies the frequency multiset in
each dyadic block with a shifted ruler multiset. The fold retains enough
order to locate its two largest values (Corollary~\ref{cor:block-peaks}).
The positive arch maxima give a sequence of adjacent index pairs on which
at least one of the two defects attains order \(n/\sqrt{\log n}\)
(Proposition~\ref{prop:defect-order}).

Two profile results bear on the defect. Let \(\mathscr D_r(t)\) count
ones minus zeros in the first \(t\) letters of \(N_r\). On the horizontal
scale \(t/|N_r|\), the step profile with values
\((\sqrt r/4^r)\mathscr D_r(t)\) has a weak Gaussian limit against
Lipschitz test functions \cite[Theorem~8.20, equations~(144)--(145)]{mantovanelli2026zeta}.
For the elementary words \(A_j\) of \cite[equation~(9.1)]{mantovanelli2026oddarch},
\(A_{2r+4}=P_r\) and \(A_{2r+5}=0N_r\) for \(r\geq0\). Write \(h_W\)
for the linearly interpolated prefix height of a word \(W\), counting
zeros minus ones. The prefixed zero gives \(h_{0N_r}(t+1)=1-\mathscr D_r(t)\)
for \(t=0,\ldots,|N_r|\). With \(\ell_j=|A_j|\), the uniform
limit of \cite[Theorem~4.1, equation~(4.3)]{mantovanelli2026oddarch} is
\[
 \lim_{j\to\infty}\ \sup_{0\leq x\leq1}
 \left|\frac{\sqrt j}{\ell_j}h_{A_j}(\ell_j x)-J(x)\right|=0.
\]
Here \(J(x)=\varphi(\Phi^{-1}(x))\) for \(0<x<1\) and \(J(0)=J(1)=0\),
where \(\varphi\) and \(\Phi\) are the standard normal density and
distribution function.

Three further questions concern the order of the frequencies and the
interaction between adjacent arch scales.
\begin{enumerate}[leftmargin=2.2em]
\item Find a direct description of \(F(s)\) inside a dyadic block, for
      example in terms of the binary digits of \(s\).
\item Evaluate \(\limsup_{k\to\infty} M_k\) and decide whether \(M_k\)
      converges.
\item Determine the limiting profiles of the even and odd values of
      \(R(n)\) on dyadic blocks. This requires matching the arch
      coordinates of \(n\) and \(2n\) in \(R(n)=E(2n)-2E(n)\).
\end{enumerate}

\section*{Data and code availability}

The two verification programs accompany this paper as arXiv ancillary files.
The Zenodo DOI for software version~1.1.0 is
\href{https://doi.org/10.5281/zenodo.24408903}{10.5281/zenodo.24408903}~\cite{mantovanelli2026code}.
The program \path{frequency_law_checks.py} checks both peak positions in
Corollary~\ref{cor:block-peaks}.

\section*{Acknowledgments}

The author gratefully acknowledges Beno{\^\i}t Cloitre's foundational proof
of global well-definedness and the structural arch-and-forest theory on which
the present frequency argument builds.

\paragraph{Disclosure of generative AI use.}
OpenAI's ChatGPT assisted with drafting and revision, LaTeX preparation,
computational checks, and the development and checking of mathematical
arguments. The author independently verified the mathematical statements,
proofs, computations, and citations and takes full responsibility for the
final manuscript.

\appendix
\section{Finite verification protocol}
\label{app:verification}

Run the commands below from the root of the extracted arXiv source archive.
The program \path{anc/frequency_law_checks.py} checks indices and
orientations as follows.
\begin{enumerate}[leftmargin=2.2em]
\item It generates \(Q\) until both slow subsequences have passed every level
      being counted, so no frequency is truncated.
\item In each tested block it checks the frequency multiset against the
      ruler multiset. It also locates the two highest frequencies in the
      ordered counts and checks their levels against \(s_k\) and \(s_k+1\)
      for \(k\geq2\), with separate checks on \(\Bcal_0,\Bcal_1\).
\item It constructs \(P_r,N_r\) directly from the two exact interleavings,
      reconstructs zero gaps, and checks
      Lemmas~\ref{lem:interior-plateau} and \ref{lem:boundary-plateau}
      with their displayed order and endpoint values.
\item It checks the exact gap transform \eqref{eq:degree-transform}, every
      reflection pair and central pair in Lemma~\ref{lem:paired-degrees}, the
      canonical tree decompositions, and the odd/even forest-annulus
      multisets as redundant diagnostics.
\end{enumerate}
The command
\begin{verbatim}
python anc/frequency_law_checks.py --max-k 18 --max-arch 8
\end{verbatim}
checks all dyadic frequency blocks through \(\Bcal_{18}\) and all orbital word
identities through arch \(r=8\).  The script reports the largest generated
index, making completeness of the finite frequency counts auditable.
The separate command
\begin{center}
\texttt{python anc/verify\_numerics.py}
\end{center}
reproduces Table~\ref{tab:defect} and checks every seed status reported in
Table~\ref{tab:seeds} through \(10^6\).

\bibliographystyle{plain}
\bibliography{references}

@book{hofstadter1979,
  author    = {Douglas R. Hofstadter},
  title     = {G{\"o}del, Escher, Bach: An Eternal Golden Braid},
  publisher = {Basic Books},
  year      = {1979}
}

@article{conolly1989,
  author  = {Brian W. Conolly},
  title   = {{Meta-Fibonacci} sequences},
  journal = {Fibonacci Quarterly},
  volume  = {27},
  year    = {1989},
  pages   = {298--302}
}

@article{tanny1992,
  author  = {Stephen M. Tanny},
  title   = {A well-behaved cousin of the {Hofstadter} sequence},
  journal = {Discrete Mathematics},
  volume  = {105},
  year    = {1992},
  pages   = {227--239}
}

@article{jackson2006,
  author  = {Brad Jackson and Frank Ruskey},
  title   = {{Meta-Fibonacci} sequences, binary trees and extremal compact codes},
  journal = {Electronic Journal of Combinatorics},
  volume  = {13},
  number  = {1},
  year    = {2006},
  pages   = {R26},
  doi     = {10.37236/1052}
}

@inproceedings{deugau2006,
  author    = {Chris Deugau and Frank Ruskey},
  title     = {Complete $k$-ary trees and generalized {Meta-Fibonacci} sequences},
  booktitle = {DMTCS Proceedings, Fourth Colloquium on Mathematics and Computer Science: Algorithms, Trees, Combinatorics and Probabilities},
  year      = {2006},
  pages     = {203--214},
  doi       = {10.46298/dmtcs.3514}
}

@article{fox2022,
  author  = {Nathan Fox},
  title   = {Connecting slow solutions to nested recurrences with linear recurrent sequences},
  journal = {Journal of Difference Equations and Applications},
  volume  = {28},
  number  = {11--12},
  year    = {2022},
  pages   = {1458--1491},
  doi     = {10.1080/10236198.2022.2152335}
}

@article{ruskeydeugau2009,
  author  = {Frank Ruskey and Chris Deugau},
  title   = {The combinatorics of certain $k$-ary {Meta-Fibonacci} sequences},
  journal = {Journal of Integer Sequences},
  volume  = {12},
  number  = {4},
  year    = {2009},
  pages   = {Article 09.4.3}
}

@article{balamohan2008,
  author  = {B. Balamohan and A. Kuznetsov and S. Tanny},
  title   = {On the behaviour of a variant of {Hofstadter}'s $Q$-sequence},
  journal = {Journal of Integer Sequences},
  volume  = {11},
  number  = {2},
  year    = {2008},
  pages   = {Article 08.2.8}
}

@misc{oeisA394051,
  author       = {{N. J. A. Sloane and The OEIS Foundation Inc.}},
  title        = {{A394051}: {Hofstadter} {$Q$}-sequence disturbed by $(-1)^n$},
  year         = {2026},
  howpublished = {\url{https://oeis.org/A394051}},
  note         = {On-Line Encyclopedia of Integer Sequences, accessed 18 July 2026}
}

@misc{cloitre2026mantovanelli,
  author        = {Beno{\^\i}t Cloitre},
  title         = {{The Mantovanelli-Hofstadter Sequence}},
  year          = {2026},
  eprint        = {2604.06237v2},
  archivePrefix = {arXiv},
  primaryClass  = {math.NT},
  howpublished  = {\href{https://arxiv.org/abs/2604.06237v2}{arXiv:2604.06237v2 [math.NT]}},
  note          = {16 July 2026. Theorem and equation numbers cited here refer to this version},
  url           = {https://arxiv.org/abs/2604.06237v2}
}

@misc{mantovanelli2026finite,
  author        = {Marco Mantovanelli},
  title         = {{Certified Return-Word Induction for a Perturbed Hofstadter Recursion}},
  year          = {2026},
  eprint        = {2603.29622v3},
  archivePrefix = {arXiv},
  primaryClass  = {math.CO},
  howpublished  = {\href{https://arxiv.org/abs/2603.29622v3}{arXiv:2603.29622v3 [math.CO]}},
  url           = {https://arxiv.org/abs/2603.29622v3}
}

@misc{mantovanelli2026zeta,
  author        = {Marco Mantovanelli},
  title         = {{A Two-Variable Zeta Function for a Parity-Perturbed Hofstadter $Q$-Recursion: The Exceptional $t=-1$ Slice and Gaussian Boundary Layers}},
  year          = {2026},
  eprint        = {2609.02412v1},
  archivePrefix = {arXiv},
  primaryClass  = {math.NT},
  howpublished  = {\href{https://arxiv.org/abs/2609.02412v1}{arXiv:2609.02412v1 [math.NT]}},
  url           = {https://arxiv.org/abs/2609.02412v1}
}

@misc{mantovanelli2026oddarch,
  author        = {Marco Mantovanelli},
  title         = {{An Odd-Arch Family of Perturbed Hofstadter $Q$-Recursions: Gaussian Profiles and Exact Frequency Laws}},
  year          = {2026},
  eprint        = {2611.01453v1},
  archivePrefix = {arXiv},
  primaryClass  = {math.CO},
  howpublished  = {\href{https://arxiv.org/abs/2611.01453v1}{arXiv:2611.01453v1 [math.CO]}},
  url           = {https://arxiv.org/abs/2611.01453v1}
}

@misc{mantovanelli2026code,
  author       = {Marco Mantovanelli},
  title        = {{Code for A Dyadic Frequency Law for a Perturbed Hofstadter $Q$-Recursion}},
  year         = {2026},
  howpublished = {Zenodo, version 1.1.0},
  note         = {\href{https://doi.org/10.5281/zenodo.24408903}
                  {doi:10.5281/zenodo.24408903}},
  doi          = {10.5281/zenodo.24408903},
  url          = {https://doi.org/10.5281/zenodo.24408903}
}

\end{document}